\documentclass[12pt,a4paper]{article}

\usepackage[active]{srcltx}
\usepackage{amsfonts}
\usepackage{verbatim}
\usepackage{amsmath}
\usepackage{amsbsy}
\usepackage{amsxtra}
\usepackage{latexsym}
\usepackage{amssymb}
\usepackage{url}
\usepackage{cite}
\usepackage{pstricks,pst-node}
\usepackage{enumerate}
\usepackage{mathrsfs}

\usepackage[margin=2cm,nohead]{geometry}

\newtheorem{theorem}{Theorem}
\newtheorem{lemma}{Lemma}

\newtheorem{proposition}{Proposition}
\newtheorem{remark}{Remark}

\DeclareMathOperator{\cone}{cone}

\DeclareMathOperator{\argmin}{argmin}

\newcommand{\Hy}{\mathcal H}
\newcommand{\lng}{\lf\langle}
\newcommand{\rng}{\rg\rangle}
\newcommand{\lf}{\left}
\newcommand{\rg}{\right}

\newcommand{\R}{\mathbb R}
\newcommand{\V}{\mathbb V}
\newcommand{\N}{\mathbb N}
\newcommand{\bs}{\begin{smallmatrix}}
\newcommand{\es}{\end{smallmatrix}}

\newcommand{\f}{\frac}
\newcommand{\ds}{\displaystyle}

\newenvironment{proof}{{\noindent\bf Proof.}}{\hfill$\Box$\\}

\begin{document}

\title{Extended Lorentz cones and mixed complementarity problems
}
\author{S. Z. N\'emeth\\School of Mathematics, The University of Birmingham\\The Watson Building, Edgbaston\\Birmingham B15 2TT, United Kingdom\\email: s.nemeth@bham.ac.uk\and G. Zhang\\School of Mathematics, The University of Birmingham\\The Watson Building, Edgbaston\\Birmingham B15 2TT, United Kingdom\\email: gxz245@bham.ac.uk}
\date{}
\maketitle

\begin{abstract}
	In this paper we extend the notion of a Lorentz cone in a Euclidean space as follows: We divide the index set corresponding to the
	coordinates of points in two disjoint classes. By definition a point belongs to an extended Lorentz cone associated with this division, 
	if the coordinates corresponding to one class are at least as large as the norm of the vector formed by the coordinates corresponding to
	the other class. We call a closed convex set isotone projection set with respect to a pointed closed convex cone if the projection onto 
	the set is isotone (i.e., order preserving) with respect to the partial order defined by the cone. We determine the isotone projection sets with respect 
	to an extended Lorentz cone. In particular, a Cartesian product between an Euclidean space and any closed convex set in another Euclidean 
	space is such a set. We use this property to find solutions of general mixed complementarity problems recursively.
\end{abstract}

\section{Introduction}

If $K\subset\R^m$ is a closed convex cone, $K^*$ the dual cone of $K$, and $ F:\,K\to\R^m$ a mapping, then the nonlinear complementarity problem
$NCP(F,K)$ defined by $K$ and 
$ f$ is the problem of finding an $ x^*\in K$ such that $ F( x^*)\in K^*$ and $\langle x^*, F( x^*)\rangle=0$. 
Some problems of economics, physics and engineering can be modelled by complementarity problems and they occur in constraint qualifications for mathematical 
programming too\cite{FacchineiPang2003}. It is known that $ x^*$ is a solution of the nonlinear complementarity problem $NCP(F,K)$ if and only 
if $ x^*$ is a fixed point of the mapping $K\ni x\mapsto  P_K(x-F(x))$, where $ P_K$ is the
projection mapping onto $K$\cite{FacchineiPang2003}. Therefore, if the sequence $\{ x^n\}_{n\in\mathbb N}$ of the Picard iteration
\begin{equation}\label{rec}
	 x^{n+1}= P_K( x^n-F(x^n)),
\end{equation} 
is convergent to $x^*\in K$ and the mapping $ F$ is continuous, then a simple limiting process in \eqref{rec} yields that $ x^*$ is a fixed point of 
the mapping $K\ni x\mapsto P_K(x-F(x))$, or equivalently a solution of the nonlinear complementarity problem defined by $K$ and 
$F$ (see Proposition 1.5.8\cite{FacchineiPang2003}).  Therefore, several papers dealt with conditions of convergence for recursions similar to \eqref{rec}, as for example 
\cite{Auslender1976,Bertsekas1989,Iusem1997,Khobotov1987,Korpelevich1976,Marcotte1991,Nagurney1993,Sibony1970,Solodov1999,Solodov1996,Sun1996}. However, neither of these 
works used the partial ordering defined by a cone for showing the convergence of the corresponding iterative scheme. Instead, they used as a tool the Banach fixed point theorem 
and assumed Kachurovskii-Minty-Browder type monotonicity (see\cite{Kachurovskii1960,Minty1962,Minty1963,Browder1964}) and global Lipschitz
properties of $F$. We will provide conditions for the convergence of \eqref{rec} in $\R^p\times\R^q$ ($p,q>0$) in terms of the partial order 
defined by the extended Lorentz cone \eqref{extlor}, when $K=\R^p\times C$, where $C$ is a general closed convex cone in $\R^q$. Although also based on the idea of isotonicity of the projection, our results are for a much wider family of cones than the isotone projection cones considered in 
\cite{IsacNemeth1990b,IsacNemeth1990c,Nemeth2009} (nevertheless 
we acknowledge that the class of isotone projection cones contain cones important from the practical point of view, such as the monotone 
cone\cite{GuyaderJegouNemeth2012} and the monotone nonnegative cone\cite{Dattorro2005}, which cannot be written as a direct product 
$K=\R^p\times C$ with $p>0$). The isotonicity property of a projection was also used by H. Nishimura and E. A. Ok 
\cite{NishimuraOk2012} for studying the solvability of variational inequalities and related equilibrium problems. We would like to emphasize that
the ordered vector structures are becoming more and more important in studying various fixed point and related equilibrium problems 
(see the book\cite{CarlHeikilla2011} of S. Carl and S Heikkil\"a and the references therein).

The structure of the paper is as follows: In the section ``Preliminaries'' we will recall several definitions and fix the terminology. In 
particular, we will define the notion of $K$-isotone mappings with respect to a pointed closed cone $K$. 
In section 4, we will extend the notion of Lorentz cones (also called ``second order cones'' or ``icecream cones'' in the literature) and show 
in Theorem \ref{tliso} that the projection mapping $P_K$ onto $K=\R^p\times C$, where $C$ is a closed convex set (in particular any closed 
convex cone) is $L$-isotone with respect to the extended Lorentz cone $L$ \eqref{extlor}. This isotonicity property will be crucial in Theorem 
\ref{tsumm} of section 6 (and Proposition \ref{psumm} on which this theorem is based) to generate an iteration which is convergent to a solution 
of a general mixed complementarity problem (extension of the mixed complementarity problem considered by Facchinei and Pang in 
\cite{FacchineiPang2003} from the nonnegative orthant to a general closed convex cone), without any restriction on the closed convex cone 
defining this problem. Section 5 has a transitional role from the nonlinear complementarity problems to the mixed complementarity problems,
in the sense that the isotonicity properties of section 4 will be used directly in section 5 for nonlinear complementarity problems on which 
the mixed complementarity problems of section 6 are based. In Section 7 we will give an example for Theorem \ref{tsumm}. The Appendix is of 
independent interest with the purpose of convincing the reader that the family of $K$-isotone mappings (used in the condition ``$I-F$ is 
$K$-isotone`` of Proposition \ref{psumm} and in the corresponding condition of Theorem \ref{tsumm}) is very wide, and later sections can be read 
without it.

We note that Theorem \ref{tliso} determines all sets 
$K\subset\R^p\times\R^q$ ($p,q>0$) with $P_K$ $L$-isotone (where $L$ is the extended Lorentz cone defined by \eqref{extlor}), family which 
contains the sets $K=\R^p\times C$, where $C$ is a closed convex set.
Theorem \ref{tliso} is interesting in its own way and may be useful in a wider context, for more general equilibrium problems and other problems 
where isotonicity occurs or can be used as a tool. 

\section{Preliminaries}
Denote by $\N$ the set of nonnegative integers. Let $m$ be a positive integer. Identify $\R^m$ with the set of column vectors with real components.  The canonical scalar 
product in $\R^m$ is defined by 
$\lng x,y\rng=x^\top y$, for any $x,y\in\R^m$. Let $\|\cdot\|$ be the norm corresponding to the scalar product $\lng\cdot,\cdot\rng$, that is, $\|x\|=\sqrt{\lng x,x\rng}$, 
for any $x\in\R^m$. 

For any $m$ positive integer denote\[\R^m_+=\{x=(x_1,\dots,x_m)^\top\in\R^m:x_1\ge0,\dots,x_m\ge0\}\] and call it the \emph{nonnegative orthant} of $\R^m$. Let $p,q$ 
positive integers. Define the Cartesian product $\R^p\times\R^q$ as the pair of vectors $(x,u)$, where $x\in\R^p$ and $u\in\R^q$. Any vector $(x,u)\in\R^p\times\R^q$ 
can be identified with the vector $\lf(x^\top,y^\top\rg)^\top\in\R^{p+q}$. The scalar product in $\R^p\times\R^q$ is given by
\[\lng(x,u),(y,v)\rng=\lng x,y\rng+\lng u,v\rng.\] 

The \emph{affine hyperplane} with the normal $u\in\R^m\setminus\{0\}$ and through $a\in\R^m$ is the set defined by
\begin{equation}\label{hyperplane}
	\Hy(u,a)=\{x\in \R^m:\;\lng x-a,u\rng=0\}.
\end{equation}
An affine hyperplane $\Hy(u,a)$ determines two \emph{closed halfspaces} $\Hy_-(a,u)$ and $\Hy_+(u,a)$  of $\R^m$, defined by
\[\Hy_-(u,a)=\{x\in \R^m:\lng x-a,u\rng\le0\},\]
and
\[\Hy_+(u,a)=\{x\in \R^m:\lng x-a,u\rng\ge0\}.\]
An \emph{affine hyperplane} through the origin will be simply called \emph{hyperplane}.

Let $\V$ be a real vector space. A set $K\subset\V$ is called a \emph{convex cone} if it is invariant with respect to the linear structure of
$\V$, that is, $\alpha x+\beta y\in K$, whenever $x,y\in K$ and $\alpha,\beta\ge0$. It is easy to show that every convex cone is a convex set. 

A convex cone $K\subset\R^m$ which is a closed set is called a closed convex cone. A closed convex cone $K\subset\R^m$ is called \emph{pointed} 
if $K\cap(-K)=\{0\}$, where $0$ is the origin of $\R^m$, that is, the vector with all entries zero. 

Let  $K\subset\R^m$ be a closed convex cone. Then, the set\[K^*=\{x\in\R^m:\lng x,y\rng\ge0,\textrm{ }\forall y\in K\}\] is called the \emph{dual} of $K$ and it is easy to 
see that it is a closed convex cone. It is known that $(K^*)^*=K$. The closed convex cone $K$ is called \emph{subdual} if $K\subset K^*$ and
\emph{self-dual} if $K=K^*$.

Let $K\subset\R^m$ be a pointed closed convex cone. Denote $\le_K$ the partial order relation defined by $x\le_K y\iff y-x\in K$ and call it 
\emph{the partial order relation defined by $K$}. 
The relation $\le_K$ is reflexive, transitive, antisymmetric and compatible with the linear structure of $\R^m$ in the sense that $x\le_K y$ implies that 
$tx+z\le_K ty+z$, for any $z\in\R^m$ and any $t\in\R_+$. Moreover, $\le_K$ is continuous at $0$ in the sense that if $x^n\to x$ when $n\to\infty$ and $0\le_K x^n$ for any 
$n\in\N$, then $0\le_K x$. Conversely any reflexive, transitive and antisymmetric relation $\le$ which is compatible with the linear structure of $\R^m$ and it is 
continuous at $0$ is defined by a pointed closed convex cone. More specifically, $\le=\le_K$, where $K=\{x\in\R^m:0\le x\}$ is a pointed closed convex cone.
 
For any closed convex set $C$ denote by $P_C$ the \emph{metric projection mapping onto $C$} that is the mapping defined by 
\[\R^m\ni x\mapsto P_Cx=\argmin\{\|y-x\|:y\in C\}.\] It can be shown (see \cite{Zarantonello1971}) that $P_C$ is a well defined point to point mapping from $\R^m$ to 
$\R^m$. From the definition above it easily follows that 
\begin{equation}\label{protr}
	P_{y+C}x=y+P_C(x-y)
\end{equation}
for any $x,y\in\R^m$. It is also known that $P_C$ is nonexpansive (see \cite{Zarantonello1971}), that is, 
\begin{equation}\label{pronexp}
	\|P_C(x)-P_C(y)\|\le\|x-y\|, 
\end{equation}
for any $x,y\in\R^m$.

Let $K\subset\R^m$ be a pointed closed convex cone. The  mapping $F:\R^m\to\R^m$ is called \emph{$K$-isotone} if $x\le_K y$ implies 
$F(x)\le_K F(y)$. 

The nonempty closed convex set 
$C\subseteq\R^m$ is called $K$-isotone projection set if $P_C$  is $K$-isotone.  

The set $\Omega\subset\R^m$ is called $K$-bounded from below ($K$-bounded from above) if there exists a vector $y\in\R^m$ such that $y\le_K x$ ($x\le_K y$), for all 
$x\in\Omega$. In this case $y$ is called a lower $K$-bound (upper $K$-bound) of $\Omega$.  
If $y\in\Omega$, then $y$ is called the $K$-least element ($K$-greatest element) of $\Omega$.

Let $\mathcal I\subset\N$ be an unbounded set of nonnegative integers. The sequence $\{x^n\}_{n\in\mathcal I}$ is called $K$-increasing ($K$-decreasing) if 
$x^{n_1}\le_K x^{n_2}$ ($x^{n_2}\le_K x^{n_1}$), whenever $n_1\le n_2$.  

The sequence $\{x^n\}_{n\in\mathcal I}$ is called $K$-bounded from below ($K$-bounded from above) if the set $\{x^n:n\in\mathcal I\}$ is $K$-bounded from below 
($K$-bounded from above).

A closed convex cone $K$ is called regular if any $K$-increasing sequence which is $K$-bounded from above is convergent. It is easy to show that this is equivalent to the 
convergence of any $K$-decreasing sequence which is $K$-bounded from below. It is known (see \cite{McArthur1970}) that any pointed closed convex cone in $\R^m$ is regular.

\section{Isotonicity of the projection with respect to extended Lorentz cones}

For $a,b\in\R^m$ denote $a\ge b$ if all components of $a$ are at least as large as the 
corresponding components of $b$, or equivalently $b\le_{\R^m_+} a$. Let $p,q$ be positive integers.  Denote by $e\in\R^p$ the vector whose all 
components are $1$. Let 
\begin{equation}\label{extlor} 
	L=\{(x,u)\in\R^p\times\R^q:x\ge\|u\|e\}
\end{equation} 
and 
\begin{equation}\label{extld}
	M=\{(x,u)\in\R^p\times\R^q:\lng x,e\rng\ge\|u\|,x\ge0\}.
\end{equation}
\begin{proposition}\label{pd} 
	$M=L^*$.
\end{proposition}

\begin{proof}
	Let $(x,u)\in L$ and $(y,v)\in M$ be arbitrary. Then, by using the Cauchy-Schwarz inequality, we get
	\begin{eqnarray*}
		\lng (x,u),(y,v)\rng=\lng x,y\rng+\lng u,v\rng\ge\lng\|u\|e,y\rng
		+\lng u,v\rng\\=\|u\|\lng e,y\rng+\lng u,v\rng\ge\|u\|\|v\|+\lng u,v\rng\ge0.
	\end{eqnarray*}
	Hence, $M\subset L^*$. Conversely, let $(x,u)\in L^*$ be arbitrary. We have 
	$(e^i,0)\in L$. Hence $0\le\lng (x,u),(e^i,0)\rng=\lng x,e^i\rng+\lng u,0\rng=x_i$.
	Thus, $x\ge0$. We also have $(e,-u/\|u\|)\in L$. Hence 
	$0\le\lng (x,u),(e,-u/\|u\|)\rng=\lng x,e\rng-\|u\|$. Thus, $\lng x,e\rng\ge\|u\|$.
	Therefore, $(x,u)\in M$ which implies $L^*\subset M$. 
\end{proof}

\begin{remark}\label{r1}
	The extended Lorentz cone $L$ defined by \eqref{extlor} is a pointed closed convex (and hence regular) cone. The cone $L$ 
	(or $L^*$ \cite{Rockafellar1970}) is a polyhedral cone if and only if $q=1$. If $q=1$, then the minimal number of 
	generators of $L$ is 
	$(p+2)(1-\delta_{p1})+2\delta_{p1}$, where $\delta$ denotes the Kronecker symbol. If $q=1$, $p=1$, then a minimal set of generators of $L$ is $\{(1,1),(1,-1)\}$, 
	and if $q=1$, $p>1$, then a minimal set of 
	generators of $L$ is $\{(e,1),(e,-1),(e^i,0):i=1,\dots,p\}$. If $q=1$, then $L^*$ is a $p+1$ dimensional polyhedral cone with the minimal number of 
	generators $2p$ and a minimal set of generators of $L^*$ is $\{(e^i,1),(e^i,-1):i=1,\dots,p\}$. 
	If $q=1$ and $p>1$, then note that the number of generators of $L$ and $L^*$ coincide if and only if they are $2$ or $3$-dimensional cones. The cone $L$ is a 
	subdual cone and $L$ is self-dual if and only if $p=1$, that is, $L$ is the $q+1$-dimensional Lorentz cone. $L$ is a self-dual polyhedral cone if and only if 
	$p=q=1$.
\end{remark}

	We will prove only the subduality of $L$ and the condition for its self-duality, because the other assertions are easy to verify. Let $(x,u)\in L$. It is easy to 
	see that $x\ge 0$. Equation
	\eqref{extlor} multiplied scalarly by $e$ gives $\lng x,e\rng\ge p\|u\|\ge\|u\|$, which implies that $(x,u)\in M$, where $M$ is the cone given by \eqref{extld}. 
	Hence, by Proposition \ref{pd}, it follows that $(x,u)\in L^*$. In conclusion, $L$ is subdual. If $p=1$, then $L$ is the $q+1$, dimensional Lorentz cone and hence
	it is self-dual. Suppose that $p>1$. Let $u\in\R^q$ such that $1<\|u\|<p$. Then, Proposition \ref{pd} and equation \eqref{extld} implies that $(e,u)\in L^*$.
	On the other hand, equation \eqref{extlor} shows that $(e,u)\notin L$. Hence, $L$ is self-dual if and only if $p=1$. 

Consider $L$ defined by \eqref{extlor}. It is easy to see that $L$ is a pointed closed convex cone. Due to the fact that for $L$ is the $q+1$-dimensional Lorentz 
cone for $p=1$ (see Remark \ref{r1}), we will call $L$ the \emph{extended Lorentz cone}. 


Recall that an affine hyperplane $\Hy$ is called tangent to a closed convex set $C\subset\R^m$ at a point
$x\in C$ if it is the unique supporting affine hyperplane to $C$ at $x$ (see pages 100 and 169 of 
\cite{Rockafellar1970}).

The following result has been shown in \cite{NemethNemeth2012b}.

\begin{theorem}\label{fooo}
	The closed convex set $C\subset\R^m$ with nonempty interior is a $K$-isotone projection set
	if and only if it is of the form
	\begin{equation*}
		C=\cap_{i\in \N} \Hy_-(u^i,a^i),
	\end{equation*}
	where each affine hyperplane $\Hy(u^i,a^i)$ is tangent to $C$ and it is a $K$-isotone projection set.
\end{theorem}

\begin{lemma}\label{ti}
	Let $K\subset\R^m$ be a closed convex cone and $\Hy\subset\R^m$ be a hyperplane 
	with a unit normal vector $a\in\R^m$. Then, $\Hy$ is a $K$-isotone projection set
	if and only if \[\lng x,y\rng\ge\lng a,x\rng\lng a,y\rng,\] for any $x\in K$ and 
	$y\in K^*$.
\end{lemma}

\begin{proof}
	Since $P_\Hy$ is linear, it follows that $P_\Hy$ is isotone if and only if 
	\begin{equation}\label{eih}
		P_\Hy x=x-\lng a,x\rng a\in K,
	\end{equation} 
	for any $x\in K$. By the definition of the dual cone and $(K^*)^*=K$, it follows that relation 
	\eqref{eih} is equivalent to 
	\[
	\lng x,y\rng=\lng a,x\rng\lng a,y\rng+\lng x-\lng a,x\rng a,y\rng
	\ge\lng a,x\rng\lng a,y\rng,
	\] 
	for any $x\in K$ and $y\in K^*$. 
\end{proof}

The next lemma follows easily from \eqref{protr}: 

\begin{lemma}\label{leasy}
	Let $z\in\R^m$, $K\subset\R^m$ be a closed convex cone and $C\subset\R^m$ be a nonempty closed convex set. Then, $C$ is a $K$-isotone projection set if and only
	if $C+z$ is a $K$-isotone projection set.
\end{lemma}

\begin{theorem}\label{tliso}
	$\,$

	\begin{enumerate}
		\item Let $K=\R^p\times C$, where $C$ is an arbitrary nonempty closed convex set in $\R^q$ and $L$ be the extended Lorentz cone defined by 
			\eqref{extlor}. Then, $K$ is an $L$-isotone projection set. 
		\item Let $p=1$, $q>1$ and $K\subset\R^p\times\R^q$ be a nonempty closed convex set. Then, $K$ is an $L$-isotone projection set if and only if 
			$K=\R^p\times C$, for some $C\subset\R^q$ nonempty closed convex set.
		\item Let $p,q>1$, and 
			\begin{equation*}
				K=\cap_{\ell\in \N} \Hy_-(\gamma^\ell,\beta^\ell)\subset\R^p\times\R^q,
			\end{equation*}
			where $\gamma^\ell=(a^\ell,u^\ell)$ is a unit vector. Then, $K$ is an $L$-isotone projection set if and only if for each $\ell$ one of the 
			following conditions hold:
			\begin{enumerate}
				\item The vector $a^\ell=0$.
				\item The vector $u^\ell=0$, and there exists $i\ne j$ such that $a^\ell_i=\sqrt2/2$, $a^\ell_j=-\sqrt2/2$ and $a^\ell_k=0$, for any 
					$k\notin\{i,j\}$.
			\end{enumerate}
	\end{enumerate}
\end{theorem}
%

\begin{proof}
	\begin{enumerate}
		\item Suppose that $K=\R^p\times C$, where $C$ is a closed convex set in $\R^q$. Let $(x,u),(y,v)\in\R^p\times\R^q$ such that 
			$(x,u)\le_L(y,v)$. Then, the nonexpansitivity \eqref{pronexp} of the projection implies 
			\[y-x\ge\|v-u\|e\ge\|P_Cv-P_Cu\|e.\] Thus, $(y,P_Cv)-(x,P_Cu)\in L$. Hence, 
			$P_K(x,u)=(x,P_Cu)\le_L(y,P_Cv)=P_K(y,v)$.
		\item The cone becomes a Lorentz cone of dimension at least $3$. This item was proved in \cite{NemethNemeth2012a,NemethNemeth2012b}.
		\item By Theorem \ref{fooo} and Lemma \ref{leasy}, we can suppose without loss of generality that $K$ is a hyperplane. Let $\gamma=(a,u)$ be the unit 
			normal vector of $K$. Suppose that one of the following conditions hold
			\begin{enumerate}
				\item The vector $a=0$.
				\item The vector $u=0$, and there exists $i\ne j$ such that $a_i=\sqrt2/2$, $a_j=-\sqrt2/2$ and $a_k=0$, for any 
					$k\notin\{i,j\}$.
			\end{enumerate}
			We need to show that $K$ is an $L$-isotone projection set. If (a) holds, then this follows easily from item 1.  Hence, suppose that (b) holds.
			By Lemma \ref{ti} we need to show that
			\begin{equation}\label{eau}
				\lng\zeta,\xi\rng\ge\lng\gamma,\zeta\rng\lng\gamma,\xi\rng,
			\end{equation} 
			for any $\zeta:=(x,v)\in L$ and $\xi:=(y,w)\in L^*$.
			Condition \eqref{eau} is equivalent to 
			\begin{equation*}
				\lng x,y\rng+\lng v,w\rng\ge\f12(x_i-x_j)(y_i-y_j),
			\end{equation*}
			or to
			\begin{equation}\label{eau2}
				\f12(x_i+x_j)(y_i+y_j)+\sum_{k\notin\{i,j\}}x_ky_k+\lng v,w\rng\ge0.
			\end{equation}
			Hence, it is enough to show \eqref{eau2}. By $(x,u)\in L$, $(y,w)\in L^*$ and the Cauchy-Schwarz inequality, we get
			\begin{gather*}
				\f12(x_i+x_j)(y_i+y_j)+\sum_{k\notin\{i,j\}}x_ky_k+\lng v,w\rng\ge\f12(\|v\|+\|v\|)(y_i+y_j)\\+\sum_{k\notin\{i,j\}}\|v\|y_k+\lng v,w\rng
				=\|v\|\lng y,e\rng+\lng v,w\rng\ge \|v\|\|w\|+\lng v,w\rng\ge0.
			\end{gather*}

			Conversely, suppose that $K$ is an $L$-isotone projection set.  
			By Lemma \ref{ti}, condition \eqref{eau} holds.
			Let $x\in\R^p_+$ and $v\in\R^q$. Then, by \eqref{extlor}, \eqref{extld} and Proposition
			\ref{pd}, it is easy to check that $\zeta:=(\|v\|e,v)\in L$, $\xi:=(\|v\|x,-\lng e,x\rng v)\in L^*$ and $\lng\zeta,\xi\rng=0$. Hence, 
			condition \eqref{eau} implies
			\begin{equation}\label{ea}
				0\ge(\lng a,e\rng\|v\|+\lng u,v\rng)(\lng a,x\rng\|v\|-\lng e,x\rng\lng u,v\rng).
			\end{equation} 
			If in \eqref{ea} $x=e$ and we choose $v\neq 0$ such that $\lng u,v\rng=0$, then we get $0\ge\lng a,e\rng\|v\|^2$, and hence 
			$\lng a,e\rng=0$. 
			Hence, \eqref{ea} becomes
			\begin{equation}\label{ea2}
				0\ge\lng u,v\rng(\lng a,x\rng\|v\|-\lng e,x\rng\lng u,v\rng).
			\end{equation}
			First, suppose that $u\ne0$. Let $v^n\in\R^q$ be a sequence of points such that $\|v^n\|=1$, $\lng u,v^n\rng>0$ and 
			$\lim_{n\to+\infty}\lng u,v^n\rng=0$. Let $n$ be an arbitrary positive integer.
			If in \eqref{ea2} we choose $\lambda>0$ sufficiently large such that $x:=a+\lambda e\ge0$ and $v=v^n$, 
			we get \(0\ge\lng u,v^n\rng(\|a\|^2-\lambda p\lng u,v^n\rng),\) or equivalently $\|a\|^2\le\lambda p\lng u,v^n\rng$. By letting 
			$n\to+\infty$ in the last inequality, we obtain $\|a\|^2\le0$, or equivalently $a=0$. 

			Next, suppose that $u=0$.
			Let $x,y\in\R^p_+$ and $w\in\R^q$ such that $\lng x,y\rng=0$, 
			$\lng y,e\rng\ge\|w\|$. Then, by \eqref{extlor}, \eqref{extld} and Proposition \ref{pd}, it is easy to check that 
			$\zeta:=(x,0)\in L$, $\xi:=(y,w)\in L^*$ and $\lng\zeta,\xi\rng=0$. Hence, equation \eqref{eau} implies     
			\begin{equation}\label{ea3}
				0\ge\lng a,x\rng\lng a,y\rng,
			\end{equation}
			for any $x,y\in\R^p_+$ with $\lng x,y\rng=0$. Let $x=e^r$ and $y=e^s$, where $r\ne s$. Then, \eqref{ea3} becomes $a_ra_s\le0$. 
			This together with $\lng e,a\rng=0$ and $1=\|\gamma\|^2=\|a\|^2$ gives that $\exists i\ne j$ such that $a_i=\sqrt{2}/2$, 
			$a_j=-\sqrt{2}/2$ and $a_k=0$, $\forall k\notin\{i,j\}$. 

		\end{enumerate}
\end{proof}
%

\section{Complementarity problems}

Recall the notion of a complementarity problem and the corresponding Picard iteration \eqref{rec} from the Introduction.
It is natural to seek convergence conditions for $x^n$. This will be done by finding cones $L$ and conditions to be imposed on $F$ such that the sequence 
$\{x^n\}_{n\in\N}$ to be $L$-increasing and $L$-bounded 
from above. These conditions will imply that $\{x^n\}_{n\in\N}$ is convergent and its limit is a solution of $NCP(F,K)$. Denote by $I$ the 
identity mapping.  

\begin{lemma}\label{lmoniso}
	Let $K\subset\R^m$ be a closed convex cone, $F:\R^m\to\R^m$ be a continuous mapping and $L$ be a pointed closed convex cone. Consider the
	sequence $\{x^n\}_{n\in\N}$ defined by \eqref{rec}. Suppose that the mappings $P_K$ and $I-F$ are $L$-isotone, $x^0\le_L x^1$, and there
	exists a $y\in\R^m$ such that $x^n\le_L y$, for all $n\in\N$ sufficiently large. Then, $\{x^n\}_{n\in\N}$ is convergent and its limit $x^*$ is a solution of 
	$NCP(F,K)$.
\end{lemma}

\begin{proof}
	Since the mappings $P_K$ and $I-F$ are $L$-isotone, the mapping $x\mapsto P_K\circ (I-F)$ is also $L$-isotone. Then, by using \eqref{rec} and a simple inductive 
	argument, it follows that $\{x^n\}_{n\in\N}$ is $L$-increasing. Since any pointed closed convex cone in $\R^m$ is regular, $\{x^n\}_{n\in\N}$ is convergent and 
	hence its limit $x^*$ is a solution of $NCP(F,K)$. 
\end{proof}

\begin{remark}\label{rc}
	$\,$
	
	\begin{enumerate}
		\item The condition $x^0\le_L x^1$ in Lemma \ref{lmoniso} is satisfied when $x^0\in K\cap F^{-1}(-L)$. Indeed, if $x^0\in K\cap F^{-1}(-L)$, then 
			$-F(x^0)\in L$ and $x^0\in K$. Thus $x^0\le_L x^0-F(x^0)$, and hence by the isotonicity of $P_K$ we obtain $x^0=P_K(x^0)\le_L P_K(x^0-F(x^0))=x^1$.
		\item The condition $x^0\le_L x^1$ in Lemma \ref{lmoniso} is satisfied when $x^0=0$ and $-F(0)\in L$. Indeed, this is a particular case of the previous 
			item. 
	\end{enumerate}
\end{remark}

\begin{proposition}\label{psumm}
	Let $L$ be a pointed closed convex cone, $K\subset\R^m$ be a closed convex cone such that $K\cap L\ne\varnothing$ 
	and $F:\R^m\to\R^m$ be a continuous mapping. Consider the sequence $\{x^n\}_{n\in\N}$ defined by \eqref{rec}. Suppose that the mappings $P_K$ and $I-F$ are $L$-isotone and 
	$x^0=0\le_L x^1$. Let $$\Omega=K\cap L\cap F^{-1}(L)=\{x\in K\cap L:F(x)\in L\}$$ and 
	$$\Gamma=\{x\in K\cap L:P_K(x-F(x))\le_L x\}.$$ 
	Consider the following assertions:
	\begin{enumerate}[(i)]
		\item\label{i} $\Omega\ne\varnothing$,
		\item\label{ii} $\Gamma\ne\varnothing$,
		\item\label{iii} The sequence $\{x^n\}_{n\in\N}$ is convergent and its limit $x^*$ is a solution of $NCP(F,K)$. Moreover, 
			$x^*$ is the $L$-least element of $\Gamma$ and a lower $L$-bound of $\Omega$.
	\end{enumerate}
	Then, $\Omega\subset\Gamma$ and (\ref{i})$\implies$(\ref{ii})$\implies$(\ref{iii}).
\end{proposition}

\begin{proof}
	 Let us first prove that $\Omega\subset\Gamma$. Indeed, let $y\in\Omega$.  Since $P_K$ is $L$-isotone, $y-F(y)\le_L y$ implies 
	 $P_K(y-F(y))\le_L P_K(y)=y$, which shows that $y\in\Gamma$. Hence, 
	 $\Omega\subset\Gamma$. Thus, (\ref{i})$\implies$(\ref{ii}) is trivial now. 
	\medskip

	\noindent 
	(\ref{ii})$\implies$(\ref{iii}):
	\medskip

	\noindent 
	Suppose that $\Gamma\ne\varnothing$. Since the mappings $P_K$ and $I-F$ are $L$-isotone, the mapping 
	 $P_K\circ (I-F)$ is also $L$-isotone. Similarly to the proof of Lemma  \ref{lmoniso}, it can be shown that $\{x^n\}_{n\in\N}$ is 
	 $L$-increasing. Let $y\in\Gamma$ 
	be arbitrary but fixed. We have $x^0=0\le_L y$. Now, suppose that $x^n\le_L y$. Since the mapping $P_K\circ (I-F)$ is $L$-isotone, 
	$x^n\le_L y$ implies that $x^{n+1}=P_K(x^n-F(x^n))\le_L P_K(y-F(y))\le_L y$. Thus, we have by induction that $x^n\le_L y$ for all
	$n\in\N$. Then, Lemma \ref{lmoniso} implies that $\{x^n\}_{n\in\N}$ is convergent and its limit $x^*\in K\cap L$ is a solution of 
	$NCP(F,K)$. Since $x^*$ is a solution of $NCP(F,K)$, we have that $P_K(x^*-F(x^*))=x^*$ and hence $x^*\in\Gamma$. 
	Therefore, $x^*$ is the $L$-least element of $\Gamma$. Since $\Omega\subset\Gamma$, $x^*$ is a lower $L$-bound of $\Omega$.
\end{proof}

We note that from the second item of Remark \ref{rc}, it follows that condition $x^0=0\le_L x^1$ of Proposition \ref{psumm} holds if $x^0=0$ and $-F(0)\in L$. We also
remark that since the definition of $\Omega$ does not contain the projection onto $K$, (for a given $F$ and $K$) it is easier to show that $\Gamma\ne\varnothing$ by first 
showing that $\Omega\ne\varnothing$. 

\section{Mixed complementarity problems}


The following lemma extends the mixed complementarity problem in \cite{FacchineiPang2003} by replacing $\R^q_+$ with an arbitrary nonempty closed convex cone in $\R^q$.

\begin{lemma}\label{lmicp}
	Let $K=\R^p\times C$, where $C$ is an arbitrary nonempty closed convex cone in $\R^q$. Let $G:\R^p\times\R^q\to\R^p$, 
	$H:\R^p\times\R^q\to\R^q$ and $F=(G,H):\R^p\times\R^q\to\R^p\times\R^q$. Then, the nonlinear complementarity problem $NCP(F,K)$ is equivalent to the mixed 
	complementarity problem $MiCP(G,H,C,p,q)$ defined by \[G(x,u)=0,\textrm{ }C\ni u\perp H(x,u)\in C^*.\]
\end{lemma}

\begin{proof}
	It follows easily from the definition of the nonlinear complementarity problem \linebreak $NCP(F,K)$, by noting that 
	$K^*=\{0\}\times C^*$. 
\end{proof}

By using the notations of Lemma \ref{lmicp}, the Picard iteration \eqref{rec} can be rewritten as:
\begin{equation}\label{mpicard}
	\lf\{
	\begin{array}{rcl}
		x^{n+1}&=&x^n-G(x^n,u^n),\\
		u^{n+1}&=&P_C(u^n-H(x^n,u^n)),
	\end{array}
	\rg.
\end{equation}
where $G(x^n,u^n)=G(x^n,u^n)$ and $H(x^n,u^n)=H(x^n,u^n)$. Consider the partial order defined by the extended Lorentz cone defined by \eqref{extlor}. Then, we obtain the following theorem.
\begin{theorem}\label{tsumm}
	Let $K=\R^p\times C$, where $C$ is a closed convex cone, $K^*$ be the dual of $K$, $G:\R^p\times\R^q\to\R^p$ and $H:\R^p\times\R^q\to\R^q$ be continuous 
	mappings, $F=(G,H):\R^p\times\R^q\to\R^p\times\R^q$, and $L$ be the extended Lorentz cone defined by \eqref{extlor}. Let $x^0=0\in\R^p$, $u^0=0\in\R^q$ and 
	consider the sequence $\{(x^n,u^n)\}_{n\in\N}$ defined by \eqref{mpicard}. Let $x,y\in\R^p$ and $u,v\in\R^q$. Suppose that $y-x\ge\|v-u\|e$ implies 
	\[y-x-G(y,v)+G(x,u)\ge\|v-u-H(y,v)+H(x,u)\|e,\] and $x^1\ge\|u^1\|e$ (in particular this holds when $-G(0,0)\ge\|H(0,0)\|e$). 
	Let $$\Omega%
	=\{(x,u)\in\R^p\times C:x\ge\|u\|e,\textrm{ }G(x,u)\ge\|H(x,u)\|e\}$$ and $$\Gamma=\{(x,u)\in\R^p\times C:x\ge\|u\|e,\textrm{ }G(x,u)\ge\|u-P_C(u-H(x,u))\|e\}.$$ 
	Consider the following assertions:
	\begin{enumerate}[(i)]
		\item\label{i2} $\Omega\ne\varnothing$,
		\item\label{ii2} $\Gamma\ne\varnothing$,
		\item\label{iii2} The sequence $\{(x^n,u^n)\}_{n\in\N}$ is convergent and its limit $(x^*,u^*)$ is a solution of \linebreak 
			$MiCP(G,H,C,p,q)$. Moreover, $(x^*,u^*)$ is a lower $L$-bound of 
			$\Omega$ and the $L$-least element of $\Gamma$. 
	\end{enumerate}
	Then, $\Omega\subset\Gamma$ and (\ref{i2})$\implies$(\ref{ii2})$\implies$(\ref{iii2}).
\end{theorem}
\begin{proof}
	First observe that $K\cap L\ne\varnothing$. By using the definition \eqref{extlor} of the extended Lorentz cone, it is easy to verify that 
	$$\Omega=K\cap L\cap F^{-1}(L)=\{z\in K\cap L:F(z)\in L\}$$ and $$\Gamma=\{z\in K\cap L:P_K(z-F(z))\le_L z\}.$$ Let $x,y\in\R^p$ and $u,v\in C$. Since 
	$y-x\ge\|v-u\|e$ implies\[y-x-G(y,v)+G(x,u)\ge\|v-u-H(y,v)+H(x,u)\|e,\] it follows that $I-F$ is $L$-isotone. Also, $x^1\ge\|u^1\|e$ means that
	$(x^0,u^0)=(0,0)\le_L (x^1,u^1)$ (in particular if $-G(0,0)\ge\|H(0,0)\|e$, or equivalently $-F(0,0)\in L$, then by the second item of Remark \ref{rc}, it follows 
	that $(x^0,u^0)=(0,0)\le_L (x^1,u^1)$). Hence, by Theorem \ref{tliso}, Proposition \ref{psumm} (with $m=p+q$) and Lemma \ref{lmicp}, it follows that $\Omega\subset\Gamma$ and 
	(\ref{i2})$\implies$(\ref{ii2})$\implies$(\ref{iii2}). 
\end{proof}

\section{An example}

Let $L$ be the extended Lorentz cone defined by \eqref{extlor}. On the conditions of  Theorem \ref{tsumm}, suppose that $C=\{(u_1,u_2)\in\R^2:u_2 \geq u_1, u_1 \geq 0 \}$ and 
$K=\mathbb{R}^2 \times C$. Let $f_1(x,u)=1/12(x_1+\|u\|+12)$ and $f_2(x,u)=1/12(x_2+\|u\|-7.2)$.  Then it is easy to show that these two functions are $L$-monotone. Let 
$w^1=(1,1,1/6,1/3)$ and $w^2=(1,1,1/3,1/6)$ so $w^1$ and $w^2$ is in 
$L$. For any two vectors $(x,u)$ and $(y,v)$ in $\R^2\times\R^2$ with $(x,u) \leq_L (y,v)$ we have $y_1-x_1 \geq \|v-u\| \geq \|u\|-\|v\|$. 
Hence,
\[f_1(y,v)-f_1(x,u)=\f1{12}(y_1-x_1-(\|u\|-\|v\|))\geq 0.\] Similarly we can prove that if $(x,u) \leq_L (y,v)$, then $f_2(y,v)-f_2(x,u) \geq 0$.
Since $L$ is 
convex, and $w^1,w^2 \in L$, if $(x,u) \leq_L (y,v)$ holds, then \[(f_1(y,v)-f_1(x,u))w^1+ (f_2(y,v)-f_2(x,u))w^2 \in L.\] Thus, 
$ f_1(x,u)w^1+ f_2(x,u)w^2 \leq_L f_1(y,v)w^1+ f_2(y,v)w^2$. Therefore, the mapping $f_1w^1+f_2w^2$ is $L$-isotone. Hence, choose the function
\[G(x,u)=\lf(\f{11}{12}x_1-\f1{12}x_2-\f16\|u\|-\f25,-\f1{12}x_1+\f{11}{12}x_2-\f16\|u\|-\f25\rg),\] 
\[H(x,u)=\lf(u_1-\f1{72}x_1-\f1{36}x_2-\f1{24}\|u\|+\f1{30},u_2-\f1{36}x_1-\f1{72}x_2-\f1{24}\|u\|-\f7{30}\rg),\] 
so that to have
\begin{equation}\label{fdef}
(x-G,u-H)=f_1w^1+f_2w^2=\left( f_1+f_2, f_1+f_2, \frac{1}{6}f_1+\frac{1}{3}f_2,  \frac{1}{3}f_1+\frac{1}{6}f_2 \right)
\end{equation}
$L$-isotone, where $G$, $H$, $f_1$ and $f_2$ are considered at the point $(x,u)$.
It is necessary to check that all the conditions in Theorem \ref{tsumm} are satisfied. First, since 
\[-G(0,0;0,0)=(f_1(0,0;0,0)+f_2(0,0;0,0),f_1(0,0; 0,0)+f_2(0,0; 0,0))=(0.4,0.4)\] and $\|H(0,0;0,0)\|=\sqrt{2}/6$, it is clear that $-G(0,0;0,0) \geq \|H(0,0;0,0)\|e$. Next, we
will show that $\Omega$ is not empty. Consider the vector $(\bar x,\bar u)=(31,31,3,4)\in K$. Obviously, $\bar x=(31,31) \geq \sqrt{3^2+4^2}e$, and since\\
\begin{equation*}
G(31,31,3,4)=(31,31)-(f_1+f_2,f_1+f_2)=(24.6,24.6)
\end{equation*}
and
\begin{equation*}
H(31,31,3,4)=(3,4)-\left( \frac{1}{6}f_1+\frac{1}{3}f_2,\frac{1}{3}f_1+\frac{1}{6}f_2 \right)=\left(\frac{23}{15},\frac{34}{15}\right),
\end{equation*}
where the functions $f_1$ and $f_2$ are considered at the point $(\bar x,\bar u)=(31,31,3,4)$, it is straightforward to check that $G(31,31,3,4)\geq\| H(31,31,3,4)\|e$. 
Thus, $(\bar x,\bar u)\in\Omega$, which shows that $\Omega\ne\varnothing$.

Now, we begin to solve the $MiCP(G,H,C,p,q)$. Suppose that $(x,u)$ is its solution. Since $G(x,u)=0$, and
\begin{equation*}
x-G(x,u)=(f_1+f_2, f_1+f_2),
\end{equation*}
where $f_i=f_i(x,u),i=1,2$, we have $x_1=x_2=f_1+f_2$.  Moreover, since \[x_1=\f1{12}(x_1+x_2)+\f16\|u\|+0.4,\] we get
\begin{equation}\label{s1}
x_1=x_2= \frac{1}{5}\|u\|+\frac{12}{25}.
\end{equation}
The perpendicularity $u\perp H(x,u)$ implies \[\langle u,H(x,u) \rangle=u_1(u_1-\frac{1}{6}f_1-\frac{1}{3}f_2)+u_2(u_2-\frac{1}{3}f_1-\frac{1}{6}f_2)=0.\] Thus,
\begin{equation}\label{s2}
u_1^2+u_2^2=\|u\|^2=f_1\left(\frac{1}{6}u_1+\frac{1}{3}u_2\right)+f_2 \left(\frac{1}{3}u_1+\frac{1}{6}u_2 \right).
\end{equation}
We will find all nonzero solutions on the boundary of $C$.
\medskip

\noindent
Case1: $u_1=u_2,u_1>0$. Then, $\|u\|= \sqrt{2}u_1= \sqrt{2}u_2$. 
Hence, from \eqref{s2}, we get
\begin{equation*}
	2u_1= \frac{1}{2}(f_1+f_2).
\end{equation*}
By \eqref{s1},  we can conclude
\begin{equation}\label{s211}
 u_1=u_2=\frac{120+6 \sqrt{2}}{995},
\end{equation}
which implies that 
\begin{equation}\label{sol1}
(x,u)=\left( \frac{480+24\sqrt{2}}{995},\frac{480+24\sqrt{2}}{995},\frac{120+6 \sqrt{2}}{995},\frac{120+6 \sqrt{2}}{995}\right).
\end{equation}
Case 2: $u_1=0$, i.e., $\|u\|=u_2$. Equation \eqref{s2} can be transformed into
\begin{equation}\label{s222}
	u_2 \left(u_2-  \frac{1}{3}f_1-\frac{1}{6}f_2\right)=0.
\end{equation}
By using \eqref{s2} again, we get $u_2=4/15$, so $u=(0,4/15)$ and 
 \begin{equation}\label{sol2}
(x,u)=\left(\frac{8}{15},\frac{8}{15},0,\frac{4}{15}\right).
\end{equation}
If  the Picard iteration shown in \eqref{mpicard} is applied and $(0,0,0,0)$ is the starting point, then we obtain
\begin{equation}\label{egpicard}
\left\{
\begin{array}{rcl}
x^{n+1} & = & x^n-G(x^n,u^n)=(f_1(x^n,u^n)+f_2(x^n,u^n))e, \\\\
u^{n+1} & = & P_C(u^n-H(x^n,u^n))\\
        & = & \ds P_C\left(\frac{1}{6}f_1(x^n,u^n)+\frac{1}{3}f_2(x^n,u^n),  \frac{1}{3}f_1(x^n,u^n)+\frac{1}{6}f_2(x^n,u^n) \right).
\end{array}\right.
\end{equation}
So, we have $x_1^{n+1}=x_2^{n+1}$. As we start from $(0,0,0,0)$, $x_1^j=x_2^j\geq 0$ for all $j \in \mathbb{N}$. Furthermore, define the set $S$ by
\begin{equation}\label{setS}
	S=\left\{ (x,u) \in \mathbb{R}^2 \times \mathbb{R}^2 : 0\leq x_1=x_2< \frac{8}{15},\textrm{ }u_1=0,\textrm{ }0\leq u_2<\frac{4}{15} \right\}.
\end{equation}
We will prove by induction that $(x^n,u^n) \in S$, for all $n\in\N$.
We have $(x^0,u^0)=(0,0,0,0) \in S$, and we need to show that as long as $(x^n,u^n) \in S$, $(x^{n+1},u^{n+1})$ defined by \eqref{egpicard} is in $S$. 
Indeed, by using the above analysis, $x_1^n=x_2^n$. By $u_1^n=0$, $\|u^n\|=u_2^n$. If $0\leq x_1^n=x_2^n< 8/15$ and $0\leq u_2^n<4/15$, we have
\begin{eqnarray*}
0<x_1^{n+1}= x_2^{n+1}=f_1(x^n,u^n)+f_2(x^n,u^n)=\frac{1}{6}(x_1^n+ u_2^n)+ \frac{2}{5} \\<\frac{1}{6}\lf(\frac{4}{15}+\frac{8}{15}\rg)+ 
\frac{2}{5}= \frac{8}{15}.
\end{eqnarray*}
On the other hand, it can be deduced that, \\
\begin{equation*}
u^n- H(x^n,u^n)= \left(\frac{1}{24}(x_1^n+u_2^n)-\frac{1}{30},\frac{1}{24}(x_1^n+u_2^n)+\frac{7}{30}\right).
\end{equation*}
Then, the first entry of $u^n-H(x^n,u^n)$ is smaller than $(1/24)(8/15+4/15)-1/30=0$ and the second entry is positive and  smaller than $(1/24)(8/15+4/15)+7/30=4/15$.
Thus, the projection of it to $C$ must be on the line $\{(u_1,u_2): u_1=0,\textrm{ }u_2 \geq 0\}$. Moreover, $u_2^{n+1}=(u^n-H(x^n,u^n))_2<\frac{4}{15}$.
Hence, by equation \eqref{egpicard},
\begin{equation*}
	u^{n+1}=(u_1^{n+1},u_2^{n+1})=P_C(u^n-H(x^n,u^n))=\left(0,\frac{1}{3}f_1(x^n,u^n)+\frac{1}{6}f_2(x^n,u^n)\right).
\end{equation*}
Therefore, equation \eqref{egpicard} can be transformed into
\begin{equation}\label{egpicard1}
\left\{
\begin{array}{l}
\ds x_1^{n+1}=x_2^{n+1}=\frac{1}{6}\lf(x_1^n+u_2^n+\frac{12}{5}\rg)\\\\
\ds u_2^{n+1}=\frac{1}{24}\lf(x_1^n+u_2^n+\frac{28}{5}\rg)
\end{array}\right.
\end{equation}
Observing that\\
\begin{equation}\label{iter}
	x_1^{n+1}=4u_2^{n+1}-\f8{15},
\end{equation} 
and by substituting \eqref{iter} (with $n+1$ replaced by $n$) into \eqref{egpicard1}$_1$, we get $u_2^{n+1}=(5/24)u_2^n+19/90$ and 
$x_1^{n+1}=(5/24)x_1^n+19/45$. Hence,
\begin{equation}\label{egpicard2}
\left\{
\begin{array}{rl}
\ds x_1^{n+1}-\frac{8}{15}=\frac{5}{24}\lf(x_1^n-\frac{8}{15}\rg)=\lf(\frac{5}{24}\rg)^n\lf(x_1^1-\frac{8}{15}\rg),\\\\
\ds u_2^{n+1}-\frac{4}{15}=\frac{5}{24}\lf(u_2^n-\frac{4}{15}\rg)=\lf(\frac{5}{24}\rg)^n\lf(u_2^1-\frac{4}{15}\rg).
\end{array}
\right.
\end{equation}
Therefore, when $n$ goes to infinity, the sequence $(x^n,u^n)$ converges to \linebreak $(8/15,8/15,0,4/15)$ which is a solution shown in Case 2.

\section*{Conclusions}
In this paper we extended the notion of Lorentz cones and showed that the projection onto a set given as the Cartesian product between an 
Euclidean space and any closed convex set $C$ in another Euclidean space is isotone with respect to the partial order defined by an
extended Lorentz cone $L$ (or shortly is an $L$-isotone projection set). When $C$ is a closed convex cone we used this property to show a Picard 
type iteration which is convergent to a solution of a general mixed complementarity problem, and we have given some examples. 
We also determined the family of all $L$-isotone projection sets, which contain the Cartesian products described above.

In the future we plan to extend the iterative idea of this paper for more general equilibrium problems. Our iterative idea may also work when 
$C$ is a general closed convex set which is not a closed convex cone, or more generally for any $L$-isotone projection set described in Theorem 
\ref{tliso}. This would lead to particular types of variational inequalities (and other related equilibrium problems) worth to be investigated.

A more ambitious plan would be to find all pairs of closed convex cones $(K,L)$ with $L$ pointed (or more generally pairs of closed convex sets 
$(K,L)$ with $L$ a pointed closed convex cone) in a Euclidean space such that $K$ is $L$-isotone. Although this plan seems utopistic any 
positive step in this direction would reveal fundamental connections between the geometric and order structure of the Euclidean space, and
could lead to interesting applications to complementarity problems (variational inequalities).

\section*{Appendix: How large is the family of $K$-isotone mappings?}

The remaining sections can be read without this one, which is entirely for the purpose of convincing the reader that the family of $K$-isotone 
mappings which occur in the condition ``$I-F$ is $K$-isotone'' of Proposition \ref{psumm} and the corresponding condition in Theorem \ref{tsumm}
is very wide.

Let $K,S\subset\R^m$ be pointed closed convex cones such that $K\subset S$. The function $f:\R^m\to\R$ is called \emph{$K$-monotone} if $x\le_K y$ implies $f(x)\le f(y)$. 
Both the $K$-monotone functions and the $K$-isotone mappings form a cone. If $f_1,\dots,f_\ell:\R^m\to\R$ are $K$-monotone and 
$w^1,\dots,w^\ell\in K$, then it is easy to see that the mapping $F:\R^m\to\R^m$  defined by 
\begin{equation}\label{elcomb}
	F(x)=f_1(x)w^1+\dots+f_\ell(x)w^\ell
\end{equation}
is $K$-isotone. It is obvious that any $S$-monotone function is also
$K$-monotone. Hence, if $f_1,\dots,f_\ell:\R^m\to\R$ are $S$-monotone, then the mapping $F$ defined by \eqref{elcomb} is $K$-isotone.  The pointed closed convex cone $S$ 
is called simplicial if there exists linearly independent vectors $u^1,\dots,u^m\in\R^m$ such that 
\begin{equation}\label{esimp}
	S=\cone\{u^1,\dots,u^m\}:=\{\lambda_1u^1+\dots+\lambda_mu^m:\lambda_1,\dots,\lambda_m\ge0\}.
\end{equation}
The vectors $u^1,\dots,u^m$ are called the \emph{generators} of $S$ and we say that $S$ is \emph{generated} by $u^1,\dots,u^m$.
It can be shown that the dual $S^*$ of a simplicial cone $S$ is simplicial. Moreover, if $U:=(u^1,\dots,u^m)$ (that is an $m\times m$ matrix with columns 
$u^1,\dots,u^m$) and $(U^\top)^{-1}=(v^1,\dots,v^m)$, then 
$S^*=\cone\{v^1,\dots,v^m\}$ \cite{BoydVandenberghe2004}. 
Let $\{e^1,e^2,\dots,e^m\}$ be the set of standard unit vectors in $\R^m$. The cone 
$\R^m_+=\{\lambda_1e^1+\dots+\lambda_me^m:\lambda_1,\dots,\lambda_m\ge0\}$ is called the nonnegative orthant. Let $S$ be the simplicial cone defined by \eqref{esimp}. If 
$f:\R^m\to\R$ is $\R^m_+$-monotone, then $\hat f:\R^m\to\R$ defined by $\hat f(x_1u^1+\dots+x_mu^m)=f(x_1e^1+\dots+x_me^m)$ is $S$-monotone. If 
$g_1,\dots,g_m:\R\to\R$ are monotone increasing, then obviously $g:\R^m\to\R$ defined by 
\begin{equation}\label{eg}
	g(x_1u^1+\dots+x_mu^m)=g_1(x_1)+\dots+g_m(x_m)
\end{equation}
is $S$-monotone. Moreover, if $f:\R^m\to\R$ is $S$-monotone and $\psi:\R\to\R$ is monotone increasing, then it is straightforward to see that $\psi\circ f$ is also
$S$-monotone. Hence, if all mappings $f_i$ in \eqref{elcomb} are formed by using a combination of \eqref{eg}, the previous property and the conicity of the $S$-monotone
functions, then the mapping $F$ defined by \eqref{elcomb} is $K$-isotone for any pointed closed convex cone $K$ contained in $S$. For any such 
cone $K$ it is easy to construct a 
simplicial cone $S$ which contains $K$. From the definition of the dual of a cone it follows that $\R^m=\{0\}^*=(K\cap(-K))^*=K^*+(-K)^*=K^*-K^*$. Thus, the smallest 
linear subspace of $\R^m$ containing $K^*$ is $\R^m$ and hence the interior of $K^*$ is nonempty (see \cite{Rockafellar1970}). Therefore, there exist $m$ linearly independent vectors in $K^*$, that 
is, $K^*$ contains a simplicial cone $T$. Let $S$ be the dual of $T$. Then, obviously $K\subset S$.

The above constructions show that for any pointed closed convex cone the family of $K$-isotone mappings, used in Proposition \ref{psumm} and 
Theorem \ref{tsumm} is very wide. Moreover, there may be many $K$-isotone mappings which are not of the above type. This topic is worth to be 
investigated in the future.

\section*{Acknowledgments}

The authors are grateful for the referee's comments which contributed to the quality of this paper.

\bibliographystyle{hieeetr}
\bibliography{bib}

\end{document}